\documentclass[numbook,envcountsame,smallcondensed]{amsart}

\usepackage{amssymb,amsmath,amsfonts,cite,color}

\numberwithin{equation}{section}

\DeclareMathOperator{\con}{con}
\DeclareMathOperator{\simple}{sim}
\DeclareMathOperator{\mul}{mul}
\DeclareMathOperator{\var}{var}
\DeclareMathOperator{\dualM}{\overleftarrow{\mathbf M}}

\newtheorem{theorem}{Theorem}[section]

\newtheorem{lemma}[theorem]{Lemma}
\newtheorem{remark}[theorem]{Remark}
\newtheorem{corollary}[theorem]{Corollary}
\theoremstyle{definition}
\newtheorem{example}[theorem]{Example}
\newtheorem{question}[theorem]{Question}

\makeatletter

\renewcommand*\subjclass[2][2010]{\def\@subjclass{#2}\@ifundefined{subjclassname@#1}{\ClassWarning{\@classname}{Unknown edition (#1) of Mathematics Subject Classification; using '2010'.}}{\@xp\let\@xp\subjclassname\csname subjclassname@#1\endcsname}}

\renewcommand{\subjclassname}{\textup{2010} Mathematics Subject Classification}

\makeatother

\begin{document}

\title[On the ascending and descending chain conditions]{On the ascending and descending chain conditions in the lattice of monoid varieties}

\thanks{Supported by the Ministry of Education and Science of the Russian Federation (project 1.6018.2017/8.9) and by Russian Foundation for Basic Research (grant 17-01-00551).}

\author{S.\,V.\,Gusev}

\address{Ural Federal University, Institute of Natural Sciences and Mathematics, Lenina 51, 620000 Ekaterinburg, Russia}

\email{sergey.gusb@gmail.com}

\date{}

\begin{abstract}
In this work we consider monoids as algebras with an associative binary operation and the nullary operation that fixes the identity element. We found an example of two varieties of monoids with finite subvariety lattices such that their join covers one of them and has a continuum cardinality subvariety lattice that violates the ascending chain condition and the descending chain condition.
\end{abstract}

\keywords{monoid, variety, lattice of varieties, ascending chain condition, descending chain condition}

\subjclass{Primary 20M07, secondary 08B15}

\maketitle

\section{Introduction and summary}
\label{introduction}

This paper is devoted to the examination of the lattice $\mathbb{MON}$ of all monoid varieties (referring to monoid varieties, we consider monoids as algebras with an associative binary operation and the nullary operation that fixes the identity element). There are a lot of articles about the monoid varieties. However, these articles are devoted mainly to the examination of identities of monoids. At the same time, although the first results about the lattice $\mathbb{MON}$ was found a long time ago (see~\cite{Head-68,Pollak-81,Wismath-86}), only a little information was recently known about the lattice $\mathbb{MON}$.{\sloppy

}
The situation has recently changed. There are papers, devoted to examination of identities of monoids, that contain also some non-trivial results about the lattice $\mathbb{MON}$ (see~\cite{Jackson-05,Lee-08,Lee-12,Lee-14}, for instance). Several works devoted to the examination of the lattice $\mathbb{MON}$ was published in 2018~\cite{Gusev-18-IzVUZ,Gusev-18-AU,Gusev-Vernikov-18}. In these articles, several restrictions on the lattices of monoid varieties formulated in terms that are somehow connected with lattice identities were studied. At the same time, when studying lattices of varieties of algebras of various types, much attention has been also paid to the \emph{finiteness conditions}, i.e., conditions that hold in every finite lattice (see~\cite[Section~10]{Shevrin-Vernikov-Volkov-09}, for instance). 

The subvariety lattice of a variety $\mathbf V$ is denoted by $L(\mathbf V)$. A variety is called \emph{finitely generated} if it is generated by a finite algebra. In~\cite[Theorems~2.3,~2.4 and~2.5]{Lvov-73} I.V.L'vov proved that for an associative ring variety $\mathbf V$, the following are equivalent:
\begin{itemize}
\item[a)] the lattice $L(\mathbf V)$ is finite;
\item[b)] the lattice $L(\mathbf V)$ satisfies the ascending chain condition;
\item[c)] $\mathbf V$ is finitely generated.
\end{itemize}
A similar result does not hold for group varieties. Indeed, there are only countably many finitely generated group varieties. At the same time, there are uncountably many periodic non-locally finite varieties of groups with subvariety lattice isomorphic to the 3-element chain~\cite{Kozhevnikov-12}. As far as we know, in the group case the question about the equivalence of the claims a) and b) still remains open. In the semigroup case the claim c) is not equivalent to the claims a) and b). This follows from the folklore fact that the 2-element semigroup with zero multiplication with a new identity element adjoined generates a semigroup variety with countably infinitely many subvarieties~\cite[Fig.~5(b)]{Evans-71}). In the semigroup case the claims a) and b) are not equivalent too. This fact follows from the results of the work~\cite{Sapir-91} where an example of semigroup varieties whose subvariety lattice satisfies the descending chain condition but violates the ascending chain condition is given. Also, this example shows that the classes of semigroup varieties whose subvariety lattices are finite or satisfy the descending chain condition are not closed with respect to the join of the varieties and to coverings. The similar questions about the class of semigroup varieties whose subvariety lattices satisfy the ascending chain condition remain open.

In~\cite[Subsection~3.2]{Jackson-Lee-18} two monoid varieties $\mathbf U$ and $\mathbf W$ are exhibited such that the subvariety lattices of both varieties are finite, while the lattice $L(\mathbf U\vee \mathbf W)$ is uncountably infinite and does not satisfy the ascending chain condition. Moreover, it follows from the proof of Theorem~3.4 in~\cite{Jackson-Lee-18} that $L(\mathbf U\vee \mathbf W)$ violates the descending chain condition. So, the classes of monoid varieties whose subvariety lattices are finite, satisfy the descending chain condition or satisfy the ascending chain condition are not closed with respect to the join of the varieties. At the same time, the results of~\cite{Jackson-Lee-18} leave open the question about stability of the these classes of varieties with respect to coverings.

In this work we find two monoid varieties with finite subvariety lattices such that their join covers one of them and has a continuum cardinality subvariety lattice that violates the ascending chain condition and the descending chain condition. Thus, we give a negative answer to the question noted in the previous paragraph.

In order to formulate the main result of the article, we need some definitions and notation. We denote by $F$ the free semigroup over a countably infinite alphabet. As usual, elements of $F$ and the alphabet are called \emph{words} and \emph{letters} respectively. The words and the letters are denoted by small Latin letters. However, the words unlike the letters are written in bold. The symbol $F^1$ stands for the semigroup $F$ with a new identity element adjoined. We treat this identity element as the empty word and denote it by $\lambda$. Expressions like $\mathbf u\approx\mathbf v$ are used for identities, whereas $\mathbf{u=v}$ means that the words $\mathbf u$ and $\mathbf v$ coincide. One can introduce notation for the following three identities:
$$
\sigma_1:\ xyzxty\approx yxzxty,\quad\sigma_2:\ xtyzxy\approx xtyzyx,\quad\sigma_3:\ xzxyty \approx xzyxty.
$$
Note that the identities $\sigma_1$ and $\sigma_2$ are dual to each other. A letter is called \emph{simple} [\emph{multiple}] \emph{in a word} $\mathbf w$ if it occurs in $\mathbf w$ once [at least twice]. Note also that the identity $\sigma_1$ [respectively, $\sigma_2$] allows us to swap the adjacent non-latest [respectively, non-first] occurrences of two multiple letters, while the identity $\sigma_3$ allows us to swap a non-latest occurrence and a non-first occurrence of two multiple letters whenever these occurrences are adjacent to each other. Put
$$
\Phi=\{x^2y\approx yx^2,\,x^2yz\approx xyxzx,\,\sigma_3\}.
$$
The trivial variety of monoids is denoted by $\mathbf T$, while $\mathbf{SL}$ denotes the variety of all semilattice monoids. For an identity system $\Sigma$, we denote by $\var\,\Sigma$ the variety of monoids given by $\Sigma$. Let us fix notation for the following varieties:
\begin{align*}
&\mathbf C=\var\{x^2\approx x^3,\,xy\approx yx\},\enskip\mathbf D_1=\var\{x^2\approx x^3,\,x^2y\approx xyx\approx yx^2\},\\
&\mathbf D_2=\var\{\Phi,\,\sigma_1,\,\sigma_2\},\enskip\mathbf M=\var\{\Phi,\,xyzxy\approx yxzxy,\,\sigma_2\},\enskip\mathbf N=\var\{\Phi,\,\sigma_2\}.
\end{align*}
If $\mathbf V$ is a monoid variety then we denote by $\overleftarrow{\mathbf V}$ the variety \emph{dual to} $\mathbf V$, i.e., the variety consisting of monoids antiisomorphic to monoids from $\mathbf V$.

The main result of the paper is the following

\begin{theorem}
\label{main result}
\begin{itemize}
\item[\textup{(i)}] The variety $\mathbf N\vee \overleftarrow{\mathbf M}$ covers the variety $\mathbf N$.
\item[\textup{(ii)}] The interval $[\mathbf M\vee\overleftarrow{\mathbf M},\mathbf N\vee\overleftarrow{\mathbf M}]$ of the lattice $L(\mathbf N\vee \overleftarrow{\mathbf M})$  contains continuum many subvarieties and does not satisfy the ascending chain condition and the descending chain condition.
\end{itemize}
\end{theorem}

The proof of Theorem~\ref{main result} implies that the lattice $L(\mathbf N\vee\overleftarrow{\mathbf M})$ ''modulo'' the interval $[\mathbf M\vee\overleftarrow{\mathbf M},\mathbf N\vee\overleftarrow{\mathbf M}]$ has the form shown in Fig.~\ref{L(N vee dual M)}.

\begin{figure}[htb]
\unitlength=1mm
\linethickness{0.4pt}
\begin{center}
\begin{picture}(55,85)
\put(35,5){\circle*{1.33}}
\put(35,15){\circle*{1.33}}
\put(35,25){\circle*{1.33}}
\put(35,35){\circle*{1.33}}
\put(35,45){\circle*{1.33}}
\put(45,55){\circle*{1.33}}
\put(25,55){\circle*{1.33}}
\put(35,65){\circle*{1.33}}
\put(15,65){\circle*{1.33}}
\put(25,75){\circle*{1.33}}

\qbezier(25,75)(35,75)(35,65)
\qbezier(25,75)(25,65)(35,65)


\put(35,5){\line(0,1){40}}
\put(35,45){\line(-1,1){20}}
\put(35,45){\line(1,1){10}}
\put(25,55){\line(1,1){10}}
\put(25,55){\line(1,1){10}}
\put(15,65){\line(1,1){10}}
\put(45,55){\line(-1,1){10}}

\put(35,2){\makebox(0,0)[cc]{\textbf T}}
\put(37,15){\makebox(0,0)[lc]{\textbf{SL}}}
\put(37,25){\makebox(0,0)[lc]{$\mathbf C$}}
\put(37,35){\makebox(0,0)[lc]{$\mathbf D_1$}}
\put(37,45){\makebox(0,0)[lc]{$\mathbf D_2$}}
\put(23,55){\makebox(0,0)[rc]{$\mathbf M$}}
\put(13,65){\makebox(0,0)[rc]{$\mathbf N$}}
\put(47,55){\makebox(0,0)[lc]{$\overleftarrow{\mathbf M}$}}
\put(37,65){\makebox(0,0)[lc]{$\mathbf M\vee\overleftarrow{\mathbf M}$}}
\put(25,79){\makebox(0,0)[cc]{$\mathbf N\vee\overleftarrow{\mathbf M}$}}
\end{picture}
\end{center}
\caption{The lattice $L(\mathbf N\vee \dualM)$}
\label{L(N vee dual M)}
\end{figure}

We note also that one of the main goals of the work~\cite{Jackson-Lee-18} is to construct several examples of finitely generated monoid varieties with continuum many subvarieties. It is verified in Erratum to~\cite{Jackson-05} that $\overleftarrow{\mathbf M}$ and $\overleftarrow{\mathbf N}$ are finitely generated. Therefore, the variety $\mathbf N\vee\overleftarrow{\mathbf M}$ is finitely generated too. So, Theorem~\ref{main result} gives a new example of finitely generated variety of monoids with continuum many subvarieties. Besides that, this theorem provides some more new examples of finitely generated varieties of monoids with continuum many subvarieties (see~Corollary~\ref{D_2 join G} below). Since the variety $\mathbf N$ is finitely generated, this variety is locally finite. Moreover, $\mathbf N$ is finitely based and has only finite many subvarieties, i.e., it is a Cross variety. It follows that the covering of a Cross monoid variety can have a continuum cardinality subvariety lattice that violates the ascending chain condition and the descending chain condition. Since the cover $\mathbf M\vee\overleftarrow{\mathbf M}$ of the Cross varieties $\mathbf M$ and $\overleftarrow{\mathbf M}$ is non-finitely based~\cite{Jackson-05}, the class of Cross monoid varieties is not closed with respect to the formation of joins and of covers.

The article consists of three sections. Section~\ref{preliminaries} contains definitions, notation and auxiliary results, while Section~\ref{proof of theorem} is devoted to the proof of Theorem~\ref{main result}.

\section{Preliminaries}
\label{preliminaries}

\subsection{A useful construction}
\label{useful construction}

The following notion was introduced by Perkins~\cite{Perkins-69} and often appeared in the literature (see~\cite{Gusev-Vernikov-18,Jackson-05,Jackson-Lee-18,Jackson-Sapir-00}, for instance). Let $W$ be a set of possibly empty words. We denote by $\overline W$ the set of all subwords of words from $W$ and by $I\bigl(\,\overline W\,\bigr)$ the set $F^1 \setminus \overline W$. It is clear that $I\bigl(\,\overline W\,\bigr)$ is an ideal of $F^1$. Then $S(W)$ denotes the Rees quotient monoid $F^1/I\bigl(\,\overline W\,\bigr)$. If $W=\{{\bf w}_1,{\bf w}_2,\dots,{\bf w}_k\}$ then we will write $S\bigl({\bf w}_1,{\bf w}_2,\dots,{\bf w}_k\bigr)$ rather than $S\bigl(\{{\bf w}_1,{\bf w}_2,\dots,{\bf w}_k\}\bigr)$. A word \textbf w is called an \emph{isoterm} for a class of semigroups if no semigroup in the class satisfies any non-trivial identity of the form $\mathbf w\approx\mathbf w'$. 

\begin{lemma}
\label{S(W) in V}
Let $\mathbf V$ be a monoid variety and $W$ a set of possibly empty words. Then $S(W)$ lies in $\mathbf V$ if and only if each word in $W$ is an isoterm for $\mathbf V$.
\end{lemma}
 
\begin{proof}
It is easy to verify that it suffices to consider the case when $W$ consists of one word (see the paragraph after Lemma~3.3 in~\cite{Jackson-05}). Then necessity is obvious, while sufficiency is proved in~\cite[Lemma~5.3]{Jackson-Sapir-00}.
\end{proof}

The following statement is dual to Proposition~1 in Erratum to~\cite{Jackson-05}.
\begin{lemma}
\label{generator of M}
The variety $\mathbf M$ is generated by monoid $S(xysxty)$.\qed
\end{lemma}

\subsection{Word problems for the varieties $\mathbf M$ and $\mathbf N$}
\label{word problems}

We introduce a series of new notions and notation. The set of all simple [multiple] letters in a word \textbf w is denoted by $\simple(\mathbf w)$ [respectively $\mul(\mathbf w)$]. The \emph{content} of a word \textbf w, i.e., the set of all letters occurring in $\bf w$, is denoted by $\con({\bf w})$. For a word $\mathbf w$ and letters $x_1,x_2,\dots,x_k\in \con(\mathbf w)$, let $\mathbf w(x_1,x_2,\dots,x_k)$ denotes the word obtained from $\mathbf w$ by retaining the letters $x_1,x_2,\dots,x_k$. 

Let $\mathbf w$ be a word and $\simple(\mathbf w)=\{t_1,t_2,\dots, t_m\}$. We can assume without loss of generality that $\mathbf w(t_1,t_2,\dots, t_m)=t_1t_2\cdots t_m$. Then $\mathbf w = \mathbf w_0 t_1 \mathbf w_1 \cdots t_m \mathbf w_m$ where $\mathbf w_0,\mathbf w_1,\dots,\mathbf w_m$ are possibly empty words and $t_0=\lambda$. The words $\mathbf w_0$, $\mathbf w_1$, \dots, $\mathbf w_m$ are called \emph{blocks} of $\mathbf w$, while $t_0,t_1,\dots,t_m$ are said to be \emph{dividers} of $\mathbf w$. The representation of the word \textbf w as a product of alternating dividers and blocks, starting with the divider $t_0$ and ending with the block $\mathbf w_m$ is called a \emph{decomposition} of the word \textbf w. A block of the word $\mathbf w$ is called a $k$-\emph{block} if this block consists of $k$th occurrences of letters in $\mathbf w$. If every block of the word $\mathbf w$ is either $1$-block or $2$-block then we say that the word $\mathbf w$ is \emph{reduced}. Recall that a word \textbf w is called \emph{linear} if every letter from $\con(\mathbf w)$ is simple in $\mathbf w$. We note that if $\mathbf w$ is a reduced word and $x$ is a multiple letter in $\mathbf w$ then $x$ cannot occur twice in the same block of $\mathbf w$. In other words, the following is true.
\begin{remark}
\label{blocks of a reduced word}
Every block of a reduced word is a linear word.
\end{remark} 

Further, let $\mathbf w$ be a reduced word. Let us consider an arbitrary $1$-block $\mathbf w_i$ of $\mathbf w$. The maximal subwords of this block consisting of the letters whose second occurrences in $\mathbf w$ lie in the same $2$-block of $\mathbf w$ are called \emph{subblocks of $1$-block} $\mathbf w_i$ of the word $\mathbf w$. The representation of $1$-block as a product of subblocks is called a \emph{decomposition} of this $1$-block. The notions of $2$-blocks and $2$-decompositions of the word $\mathbf w$ are defined dually. The representation of the reduced word $\mathbf w$ as a product of alternating dividers and decompositions of blocks is called \emph{full decomposition} of $\mathbf w$. Below we underline dividers to distinguish them from blocks and we divide subblocks of the same block by the symbol ''$|$''. We illustrate the introduced notions by the following

\begin{example}
\label{example blocks}
Put $\mathbf w=abcdxcbyezaed$. Clearly, $\simple(\mathbf w)=\{x,y,z\}$. Therefore, the letters $x$, $y$ and $z$ are the dividers of $\mathbf w$, while the words $abcd$, $cb$, $e$ and $aed$ are the blocks of $\mathbf w$. Evidently, the blocks $abcd$ and $e$ consist of the first occurrences of letters in $\mathbf w$, while the blocks $cb$ and $aed$ consist of the second occurrences of letters in $\mathbf w$. Consequently, $abcd$ and $e$ are the $1$-blocks of the word $\mathbf w$, while $cb$ and $aed$ are the $2$-blocks of this word. So, the word $\mathbf w$ is reduced. The second occurrences of the letters $a$, $b$ and the letters $c$, $d$ lie in the different $2$-blocks, while the second occurrences of the letters $b$ and $c$ lie in the same $2$-blocks. Therefore, the decomposition of $1$-block $abcd$ has the form $a\,|\,bc\,|\,d$. The decomposition of $2$-block $cb$ consist of one subblock because the first occurrences of letters $b$ and $c$ lie in the same $1$-block. Finally, the decomposition of $2$-block $aed$ equals $a\,|\,e\,|\,d$ because the first occurrences of the letters $a$, $d$ and the letter $e$ lie in different $1$-blocks. Thus, the full decomposition of the word $\mathbf w$ has the form $a\,|\,bc\,|\,d\underline x\,cb\,\underline y\,e\,\underline z\,a\,|\,e\,|\,d$.
\end{example}

An identity $\mathbf u\approx \mathbf v$ is called \emph{reduced} if the words $\mathbf u$ and $\mathbf v$ are reduced. The words $\mathbf u$ and $\mathbf v$ are said to be \emph{equivalent} if their decompositions are
\begin{align}
\label{decomposition of u}
&\mathbf u_0 t_1 \mathbf u_1 \cdots t_m \mathbf u_m,\\
\label{decomposition of v}
&\mathbf v_0 t_1 \mathbf v_1 \cdots t_m \mathbf v_m
\end{align}
respectively and $\con(\mathbf u_i)=\con(\mathbf v_i)$ for every $0\le i\le m$. Let $\mathbf u$ and $\mathbf v$ be equivalent reduced words. Suppose that~\eqref{decomposition of u} and~\eqref{decomposition of v} are the decompositions of these words respectively. In this case, the blocks $\mathbf u_i$ and $\mathbf v_i$ are said to be \emph{corresponding} to each other. Let's say that the corresponding blocks $\mathbf u_i$ and $\mathbf v_i$ are \emph{equivalent} if their decompositions are
\begin{align}
\label{decomposition of u_i}
&\mathbf u_{i1}\mathbf u_{i2}\dots \mathbf u_{ik_i},\\
\label{decomposition of v_i}
&\mathbf v_{i1}\mathbf v_{i2}\dots \mathbf v_{ik_i}
\end{align}
respectively and $\con(\mathbf u_{ij})=\con(\mathbf v_{ij})$ for every $1\le j\le k_i$. In this case, the subblocks $\mathbf u_{ij}$ and $\mathbf v_{ij}$ are said to be \emph{corresponding} to each other. The equivalent words $\mathbf u$ and $\mathbf v$ are said to be $1$-\emph{equivalent} if every two corresponding $1$-blocks of these words are equivalent each other. 

\begin{lemma}
\label{word problem M}
A reduced identity $\mathbf u\approx \mathbf v$ holds in the variety $\mathbf M$ if and only if the words $\mathbf u$ and $\mathbf v$ are $1$-equivalent. 
\end{lemma}

\begin{proof}
\emph{Necessity}. Suppose the reduced identity $\mathbf u\approx \mathbf v$ holds in the variety $\mathbf M$.

First, we will prove that $\mathbf u$ and $\mathbf v$ are equivalent. It is proved in~\cite[Proposition 2.13]{Gusev-Vernikov-18} that an identity $\mathbf u\approx\mathbf v$ holds in the variety $\mathbf D_1$ if and only if $\simple(\mathbf u)=\simple(\mathbf v)$, $\mul(\mathbf u)=\mul(\mathbf v)$ and the simple letters appear in the words $\mathbf u$ and $\mathbf v$ in the same order. This fact and inclusion $\mathbf D_1\subset \mathbf M$ imply that if~\eqref{decomposition of u} is the decomposition of the word $\mathbf u$ then the decomposition of the word $\mathbf v$ has the form~\eqref{decomposition of v}. Let  $z\in\mul(\mathbf u)$. Suppose that, the first occurrence of $z$ in $\mathbf u$ lies in the block $\mathbf u_p$, while the second occurrence of $z$ in $\mathbf u$ lies in the block $\mathbf u_q$ for some $0\le p,q\le m$. Remark~\ref{blocks of a reduced word} implies that $p<q$. The word $ztz$ is an isoterm for the variety $\mathbf M$ by Lemmas~\ref{S(W) in V} and~\ref{generator of M}. Consequently, $\mathbf u(z,t_r)=zt_rz=\mathbf v(z,t_r)$ for every $p< r\le q$. Moreover, $\mathbf v(z,t_s)\ne zt_sz$ and $\mathbf v(z,t_\ell)\ne zt_\ell z$ for all $1\le s \le p$ and $q<\ell\le m$. Therefore, the first occurrence of $z$ in $\mathbf v$ lies in the block $\mathbf v_p$, while the second occurrence of $z$ in $\mathbf v$ lies in the block $\mathbf v_q$. This implies that the words $\mathbf u$ and $\mathbf v$ are equivalent.

Suppose now that the words $\mathbf u$ and $\mathbf v$ are not $1$-equivalent. Then there exist a $1$-block $\mathbf u_i$ of $\mathbf u$ and letters $x,y\in\con(\mathbf u_i)$ such that the second occurrences of $x$ and $y$ lie in different $2$-blocks of $\mathbf u$, say $\mathbf u_j$ and $\mathbf u_k$ respectively, while the first occurrence of $x$ precedes the first occurrence of $y$ in $\mathbf u$, but the first occurrence of $y$ precedes the first occurrence of $x$ in $\mathbf v$. In view of Remark~\ref{blocks of a reduced word}, we have that $i<j$. Besides that, we can assume without loss of generality that $j<k$. Since the words $\mathbf u$ and $\mathbf v$ are equivalent, $x\in\con(\mathbf v_i)\cap\con(\mathbf v_j)$ and $y\in\con(\mathbf v_i)\cap\con(\mathbf v_k)$. Then $\mathbf M$ satisfies the identity $\mathbf u(x,y,t_{i+1},t_{j+1})\approx \mathbf v(x,y,t_{i+1},t_{j+1})$, where
$$
\mathbf u(x,y,t_{i+1},t_{j+1})=xyt_{i+1}xt_{j+1}y\text{ and } \mathbf v(x,y,t_{i+1},t_{j+1})=yxt_{i+1}xt_{j+1}y,
$$
so that $xysxty$ is not an isoterm for $\mathbf M$. But this is impossible by Lemmas~\ref{S(W) in V} and~\ref{generator of M}.

\smallskip

\emph{Sufficiency}. Let~\eqref{decomposition of u} be the decomposition of $\mathbf u$. Since the words $\mathbf u$ and $\mathbf v$ are $1$-equivalent, the decomposition of $\mathbf v$ has the form~\eqref{decomposition of v}. We are going to verify that the word $xysxty$ is an isoterm for the variety $\var\{\mathbf u\approx \mathbf v\}$. Arguing by contradiction, we suppose that this variety satisfies a non-trivial identity $xysxty\approx \mathbf w$. We can assume without any loss that $xysxty=\mathbf a\xi(\mathbf u)\mathbf b$ and $\mathbf w=\mathbf a\xi(\mathbf v)\mathbf b$ for some words $\mathbf a$, $\mathbf b$ and some endomorphism $\xi$ of $F^1$. Since the identity $xysxty\approx \mathbf w$ is non-trivial and the words $\mathbf u$ and $\mathbf v$ are equivalent, there are letters $c,d\in\mul(\mathbf u)$ such that $x\in\con(\xi(c))$ and $y\in\con(\xi(d))$. Clearly, $c\ne d$, whence $\xi(c)=x$ and $\xi(d)=y$. This implies that $\mathbf a=\mathbf b=\lambda$ and $s=\xi(t_\ell)$, $t=\xi(t_r)$ for some $1\le \ell<r\le m$. Then, since the word $\mathbf u$ is reduced, $c\in\con(\mathbf u_i)\cap\con(\mathbf u_j)$ and $d\in\con(\mathbf u_p)\cap\con(\mathbf u_q)$ for some $i,j,p$ and $q$ such that $i\le p<\ell\le j<r\le q$. Further, we have that $c\in\con(\mathbf v_i)\cap\con(\mathbf v_j)$ and $d\in\con(\mathbf v_p)\cap\con(\mathbf v_q)$ because the words $\mathbf u$ and $\mathbf v$ are equivalent. If $i<p$ then the first occurrence of $c$ precedes the first occurrence of $d$ in $\mathbf v$. We obtain a contradiction with the fact that $\mathbf w\ne xysxty$. If $i=p$ then the first occurrence of $c$ precedes the first occurrence of $d$ in $\mathbf v$ again because the $\mathbf u$ and $\mathbf v$ are $1$-equivalent and the second occurrences of the letters $c$ and $d$ lie in different blocks of the words $\mathbf u$ and $\mathbf v$. So, the word $xysxty$ is an isoterm for $\var\{\mathbf u\approx \mathbf v\}$. In view of Lemmas~\ref{S(W) in V} and~\ref{generator of M}, the identity $\mathbf u\approx \mathbf v$ holds in $\mathbf M$.
\end{proof}

If \textbf u and \textbf v are words and $\varepsilon$ is an identity then we will write $\mathbf u\stackrel{\varepsilon}\approx\mathbf v$ in the case when the identity $\mathbf u\approx\mathbf v$ follows from $\varepsilon$.

\begin{lemma}
\label{word problem N}
A reduced identity $\mathbf u\approx \mathbf v$ holds in the variety $\mathbf N$ if and only if the words $\mathbf u$ and $\mathbf v$ are equivalent and corresponding $1$-blocks of the words $\mathbf u$ and $\mathbf v$ equal each other.
\end{lemma}

\begin{proof}
\emph{Necessity}. Suppose the reduced identity $\mathbf u\approx \mathbf v$ holds in the variety $\mathbf N$. Lemma~\ref{word problem M} and inclusion $\mathbf M\subset \mathbf N$ imply that the words $\mathbf u$ and $\mathbf v$ are $1$-equivalent. In particular, the words $\mathbf u$ and $\mathbf v$ are equivalent. Then if~\eqref{decomposition of u} is the decomposition of the word $\mathbf u$ then the decomposition of the word $\mathbf v$ has the form~\eqref{decomposition of v}. Suppose that there exist two corresponding $1$-blocks $\mathbf u_i$ and $\mathbf v_i$ such that $\mathbf u_i\ne\mathbf v_i$. Clearly, $i<m$. Since the $1$-blocks $\mathbf u_i$ and $\mathbf v_i$  are $1$-equivalent, there is a subblock $\mathbf u'$ of the block $\mathbf u_i$ which does not coincide with the corresponding subblock $\mathbf v'$ of the block $\mathbf v_i$. In view of Remark~\ref{blocks of a reduced word}, the subblocks $\mathbf u'$ and $\mathbf v'$ are linear words. At the same time, $\con(\mathbf u')=\con(\mathbf v')$ because the blocks $\mathbf u_i$ and $\mathbf v_i$ are $1$-equivalent. Therefore, there exist letters $x$ and $y$ such that $x$ precedes $y$ in $\mathbf u'$, while $y$ precedes $x$ in $\mathbf v'$. Then $\mathbf N$ satisfies the identity
$$
xyt_{i+1}\mathbf w=\mathbf u(x,y,t_{i+1})\approx \mathbf v(x,y,t_{i+1})=yxt_{i+1}\mathbf w',
$$
where $\mathbf w,\mathbf w'\in\{xy,yx\}$. Then $\mathbf N$ satisfies the identities
$$
xyt_{i+1}xy\stackrel{\sigma_2}\approx xyt_{i+1}\mathbf w\approx yxt_{i+1}\mathbf w'\stackrel{\sigma_2}\approx yxt_{i+1}xy.
$$
We obtain a contradiction with the fact that the varieties $\mathbf M$ and $\mathbf N$ are different.

\smallskip

\emph{Sufficiency}. Let~\eqref{decomposition of u} be the decomposition of $\mathbf u$. Since the word $\mathbf u$ and $\mathbf v$ are equivalent, the decomposition of $\mathbf v$ has the form~\eqref{decomposition of v}. Consider arbitrary corresponding $2$-blocks $\mathbf u_i$ and $\mathbf v_i$. The words $\mathbf u_i$ and $\mathbf v_i$ are linear (see Remark~\ref{blocks of a reduced word}) and depend on the same letters (since $\mathbf u$ and $\mathbf v$ are equivalent). The identity $\sigma_2$ allows us to swap the second occurrences of two multiple letters whenever these occurrences are adjacent to each other. Thus, if we replace the $2$-block $\mathbf u_i$ by $\mathbf v_i$ in $\mathbf u$ then the word we obtain should be equal to $\mathbf u$ in $\mathbf N$. By the hypothesis, the corresponding $1$-blocks of the words $\mathbf u$ and $\mathbf v$ equal each other. Therefore, the identity
$$
\mathbf u= t_0\mathbf u_0 t_1 \mathbf u_1 \cdots t_m \mathbf u_m\stackrel{\sigma_2}\approx  t_0\mathbf v_0 t_1 \mathbf v_1 \cdots t_m \mathbf v_m=\mathbf v,
$$
holds in the variety $\mathbf N$.
\end{proof}

\subsection{Some words and their properties}
\label{some words}

We introduce some new notation. As usual, the symbol $\mathbb N$ stands for the set of all natural numbers. For all $n,k\in \mathbb N$, we put
$$
\mathcal M_n^k=\{1\}\times\underbrace{\mathbb N_n\times \mathbb N_n \times \cdots \times \mathbb N_n}_{k-1 \text{ copies}}.
$$
where $\mathbb N_n=\{1,2,\dots, n\}$. If $\gamma=(1,i_1,i_2,\dots,i_{k-1})\in \mathcal M_n^k$ and $1\le j\le n$ then we put 
$$
\gamma+j=(1,i_1,i_2,\dots,i_{k-1},j)\in \mathcal M_n^{k+1}.
$$
The usual lexicographical order is defined on the set $\mathcal M_n^k$. So, the expression
$$
\prod_{\gamma\in \mathcal M_n^k} \mathbf w_{\gamma}
$$
mean an abbreviated notation of the product of the words $\mathbf w_{\gamma}$ in ascending order $\gamma$. Put
\begin{align*}
\mathbf c_n=\prod_{\gamma\in \mathcal M_n^n} \biggl(\prod_{j=1}^n s_{\gamma}^{(j)}x_{\gamma}^{(j)}\biggr)
&\text{ and }
\mathbf d_n^{(k)}=s_k\cdot\biggl(\prod_{\gamma\in \mathcal M_n^k} \biggl(\prod_{j=1}^n x_{\gamma}^{(j)}\biggl(\prod_{\ell=1}^n x_{\gamma+j}^{(\ell)}\biggr)\biggr)\biggr),\\[-3pt]
\mathbf e_m=s_mx_mt_my_m &\text{ and }\mathbf f_m=s_mx_mx_{m+1}y_{m+1}y_m
\end{align*}
for all $n,k\ge 1$ and $m\ge 0$. Further, for any natural $n$ put
\begin{align*}
&\mathbf a_n=xy\cdot\biggl(\prod_{i=1}^n \mathbf d_{2n}^{(2i-1)}\biggr)\cdot \mathbf c_{2n}\cdot\biggl(\prod_{i=n-1}^1 \mathbf d_{2n}^{(2i)}\biggr)\cdot sx\cdot\biggl(\prod_{i=1}^{2n} x_1^{(i)}\biggr)\cdot y,\\[-3pt]
&\mathbf a_n'=yx\cdot\biggl(\prod_{i=1}^n \mathbf d_{2n}^{(2i-1)}\biggr)\cdot \mathbf c_{2n}\cdot\biggl(\prod_{i=n-1}^1 \mathbf d_{2n}^{(2i)}\biggr)\cdot sx\cdot\biggl(\prod_{i=1}^{2n} x_1^{(i)}\biggr)\cdot y,\\[-3pt]
&\mathbf b_n=x_0y_0\biggl(\prod_{i=1}^n \mathbf f_{2i-1}\biggr)\cdot \mathbf e_{2n}\cdot\biggl(\prod_{i=n-1}^0 \mathbf f_{2i}\biggr),\\[-3pt]
&\mathbf b_n'=y_0x_0\biggl(\prod_{i=1}^n \mathbf f_{2i-1}\biggr)\cdot \mathbf e_{2n}\cdot\biggl(\prod_{i=n-1}^0 \mathbf f_{2i}\biggr).
\end{align*}

We note that the words $\mathbf a_n$, $\mathbf a_n'$, $\mathbf b_n$ and $\mathbf b_n'$ are reduced. The following simple fact can be easily verified directly.

\begin{remark}
\label{rm s-decompositions of a_n,a_n'}
The expression
\begin{equation}
\label{s-decompositions of a_n,a_n'}
\begin{aligned}
&\chi(xy)\cdot\biggl(\prod_{i=1}^{n-1}\underline{s_i} \biggl(\prod_{\gamma\in \mathcal M_{2n}^{2i-1}}  \biggl(\prod_{j=1}^{2n} x_{\gamma}^{(j)}\,\biggl|\,\prod_{\ell=1}^{2n} x_{\gamma+j}^{(\ell)}\,\biggr|\,\biggr)\biggr)\biggr)\\[-3pt]
&\cdot\biggl(\underline{s_{2n-1}} \prod_{\gamma\in \mathcal M_{2n}^{2n-1}}  \biggl(\prod_{j=1}^{2n} x_{\gamma}^{(j)}\biggl(\prod_{\ell=1}^{2n}\,|\,x_{\gamma+j}^{(\ell)}\,|\,\biggr)\biggr)\biggr)\biggl(\prod_{\gamma\in \mathcal M_{2n}^{2n}} \biggl(\prod_{j=1}^{2n} \underline{s_{\gamma}^{(j)}}x_{\gamma}^{(j)}\biggr)\biggr)\\[-3pt] 
&\cdot\biggl(\prod_{i=n-1}^1 \underline{s_i}\biggl(\prod_{\gamma\in \mathcal M_{2n}^{2i}}  \biggl(\prod_{j=1}^{2n} x_{\gamma}^{(j)}\,\biggl|\,\prod_{\ell=1}^{2n} x_{\gamma+j}^{(\ell)}\,\biggr|\,\biggr)\biggr)\biggr)\cdot \underline{s}x\cdot\,\biggl|\,\prod_{i=1}^{2n} x_1^{(i)}\biggr|\,\cdot y
\end{aligned}
\end{equation}
is the full decomposition of $\mathbf a_n$ whenever $\chi(xy)=xy$, and the full decomposition of $\mathbf a_n'$ whenever $\chi(xy)=yx$.
\end{remark}

The following two observations play an important role below.

\begin{remark}
\label{subwords of word equivalent to b_n}
Suppose that
\begin{equation}
\label{equalities for b_n}
\{\zeta_\ell(x_{2\ell+1}),\zeta_\ell(y_{2\ell+1})\}=\{x_{2\ell+1},y_{2\ell+1}\}
\end{equation}
for all $0\le \ell\le n-1$. Then every subword of length $>1$ of the word
\begin{equation}
\label{word equivalent to b_n}
x_0y_0\biggl(\prod_{i=1}^n \mathbf f_{2i-1}\biggr)\cdot \mathbf e_{2n}\cdot\biggl(\prod_{i=n-1}^0 s_{2i}x_{2i}\zeta_i(x_{2i+1})\zeta_i(y_{2i+1})y_{2i}\biggr)
\end{equation}
has exactly one occurrence in this word.
\end{remark}

\begin{remark}
\label{subwords of word equivalent to a_n}
Suppose that
\begin{equation}
\label{equalities for a_n}
\{\zeta_\alpha(x_\alpha^{(i)})\mid 1\le i\le 2n\}=\{x_\alpha^{(i)}\mid 1\le i\le 2n\}
\end{equation}
for all odd $1\le\ell\le 2n-1$ and for all $\alpha\in \mathcal M_{2n}^\ell$. Then every subword of length $>1$ of the word
\begin{equation}
\label{word equivalent to a_n}
\begin{aligned}
&xy\cdot\biggl(\prod_{i=1}^{n-1}\underline{s_i} \biggl(\prod_{\gamma\in \mathcal M_{2n}^{2i-1}}  \biggl(\prod_{j=1}^{2n} x_{\gamma}^{(j)}\,\biggl|\,\prod_{\ell=1}^{2n} x_{\gamma+j}^{(\ell)}\,\biggr|\,\biggr)\biggr)\biggr)\\[-3pt]
&\cdot\biggl(\underline{s_{2n-1}}\prod_{\gamma\in \mathcal M_n^{2n-1}}  \biggl(\prod_{j=1}^{2n} x_{\gamma}^{(j)}\biggl(\prod_{\ell=1}^{2n}\,|\,x_{\gamma+j}^{(\ell)}\,|\,\biggr)\biggr)\biggr)\cdot\biggl(\prod_{\gamma\in \mathcal M_{2n}^{2n}} \biggl(\prod_{j=1}^{2n} \underline{s_{\gamma}^{(j)}}x_{\gamma}^{(j)}\biggr)\biggr)\\[-3pt]
&\cdot\biggl(\prod_{i=n-1}^1\underline{s_i}\biggl(\prod_{\gamma\in \mathcal M_{2n}^{2i}} \biggl(\prod_{j=1}^{2n} x_{\gamma}^{(j)}\,\biggl|\,\prod_{\ell=1}^{2n} \zeta_{\gamma+j}(x_{\gamma+j}^{(\ell)})\,\biggr|\,\biggr)\biggr)\biggr)\cdot \underline{s}x\cdot\,\biggl|\,\prod_{i=1}^{2n} \zeta_1(x_1^{(i)})\,\biggr|\,\cdot y
\end{aligned}
\end{equation}
has exactly one occurrence in this word.
\end{remark}

Remarks~\ref{subwords of word equivalent to b_n} and~\ref{subwords of word equivalent to a_n} follow from the directly verifiable fact that if $\mathbf w$ is one of the words~\eqref{word equivalent to b_n} or~\eqref{word equivalent to a_n} and $ab$ is a subword of the word $\mathbf w$ then this subwords has exactly one occurrence in this word.
\smallskip

\begin{lemma}
\label{substitution a_k in b_n} 
Let $n$ be a natural number, $\xi$ be an endomorphism of $F^1$ and $\mathbf w\approx \mathbf w'$ be a non-trivial identity. Suppose that the word $\mathbf w$ coincides with the word~\eqref{word equivalent to b_n} where $\zeta_1,\zeta_2,\dots, \zeta_{n-1}$ are endomorphisms of $F^1$ such that the equality~\eqref{equalities for b_n} is true for all $0\le\ell\le n-1$. Then if $\mathbf w=\mathbf u\xi(\mathbf a_k)\mathbf v$ and $\mathbf w'=\mathbf u\xi(\mathbf a_k')\mathbf v$ for some words $\mathbf u$ and $\mathbf v$ and some $k\ge n$ then $n=k$.
\end{lemma}

\begin{proof}
We note that $\xi(x)\ne\lambda$ and $\xi(y)\ne\lambda$ because the identity $\mathbf w\approx \mathbf w'$ is non-trivial. It follows that the length of the word $\xi(xy)$ is more than $2$. In view of Remark~\ref{subwords of word equivalent to b_n}, the words $\xi(x)$ and $\xi(y)$ are letters. Remark~\ref{rm s-decompositions of a_n,a_n'} implies that the words $\mathbf a_k$ and $\mathbf a_k'$ are $1$-equivalent. According to Lemma~\ref{word problem M} the identity $\mathbf a_k\approx \mathbf a_k'$ holds in $\mathbf M$. Consequently, $\mathbf M$ satisfies the identity $\mathbf w\approx\mathbf w'$. Since the first occurrences of the letters $\xi(x)$ and $\xi(y)$ occur in the words $\mathbf w$ and $\mathbf w'$ in the opposite order, Lemma~\ref{word problem M} implies that some subblock of some $1$-block of the word $\mathbf w$ contains the subword $\xi(xy)$. It is easy to see that the full decomposition of the word $\mathbf w$ has the form
\begin{equation}
\label{decomposition of e-b_n}
\begin{aligned}
x_0y_0\cdot\biggl(\prod_{i=1}^{n-1} \underline{s_{2i-1}}x_{2i-1}\,|\,x_{2i}y_{2i}\,|\,y_{2i-1}\biggr)\cdot&\underline{s_{2n-1}}x_{2n-1}\,|\,x_{2n}\,|\,y_{2n}\,|\,y_{2n-1}\\[-3pt] 
\cdot\,\underline{s_{2n}}x_{2n}\underline{t_{2n}}y_{2n}\cdot\biggl(\prod_{i=n-1}^0 \underline{s_{2i}}x_{2i}&\,|\,\zeta_i(x_{2i+1})\zeta_i(y_{2i+1})\,|\,y_{2i}\biggr).
\end{aligned}
\end{equation}
Then $\xi(xy)=x_{2p}y_{2p}$ for some $0\le p< n$, whence $\xi(x)=x_{2p}$ and $\xi(y)=y_{2p}$. So, 
$$
\zeta_p(x_{2p+1})\zeta_p(y_{2p+1})=\prod_{i=1}^{2k} \xi(x_1^{(i)}).
$$
In view of Remark~\ref{subwords of word equivalent to b_n} and the equality~\eqref{equalities for b_n}, there are $c_1<d_1$ such that $\xi(x_1^{(c_1)})=\zeta_p(x_{2p+1})$ and $\xi(x_1^{(d_1)})=\zeta_p(y_{2p+1})$. Since the first occurrence of $x_1^{(c_1)}$ precedes the first occurrence of $x_1^{(d_1)}$ in $\mathbf a_k$, we have that the first occurrence of  $\xi(x_1^{(c_1)})$ precedes the first occurrence of  $\xi(x_1^{(d_1)})$ in $\mathbf w$, whence $\xi(x_1^{(c_1)})=x_{2p+1}$ and $\xi(x_1^{(d_1)})=y_{2p+1}$.

Further, we will prove by induction that for all $1\le q\le 2n-2p$ the equalities
\begin{equation}
\label{xi(x_2p+q) and xi(y_2p+q)}
x_{2p+q}=\xi(x_{\alpha_q}^{c_q}) \text{ and } y_{2p+q}=\xi(x_{\beta_q}^{d_q})
\end{equation}
are true for some letters $x_{\alpha_q}^{c_q}$ and $x_{\beta_q}^{d_q}$ such that $\alpha_q,\beta_q\in \mathcal M_{2k}^q$ and the first occurrence of $x_{\alpha_q}^{c_q}$ precedes the first occurrence of $x_{\beta_q}^{d_q}$ in $\mathbf w$, i.e., 
\begin{equation}
\label{condition for alpha_q beta_q}
\alpha_q\le\beta_q \text{ and if }\alpha_q=\beta_q \text{ then }c_q< d_q.
\end{equation}
The induction base is considered in the previous paragraph. Suppose that for all $1\le r< q\le 2n-2p$ there exist letters $x_{\alpha_r}^{c_r}$ and $x_{\beta_r}^{d_r}$ such that $x_{2p+r}=\xi(x_{\alpha_r}^{c_r})$ and $y_{2p+r}=\xi(x_{\beta_r}^{d_r})$, $\alpha_r,\beta_r\in \mathcal M_{2k}^r$, $\alpha_r\le\beta_r$ and if $\alpha_r=\beta_r$ then $c_r\le d_r$. We need to check that there are letters $x_{\alpha_q}^{c_q}$ and $x_{\beta_q}^{d_q}$ such that $\alpha_q,\beta_q\in \mathcal M_{2k}^q$ and the claims~\eqref{xi(x_2p+q) and xi(y_2p+q)} and~\eqref{condition for alpha_q beta_q} are true.

If $a,b\in\mul(\mathbf w)$ and $i\in\{1,2\}$ then $\mathbf w_i[a,b]$ denotes the subword of the word $\mathbf w$ located between $i$th occurrences of $a$ and $b$.

Suppose that $q$ is odd. Then 
$$
\zeta_{\frac{2p+q-1}{2}}(x_{2p+q})\zeta_{\frac{2p+q-1}{2}}(y_{2p+q})=\xi(\mathbf w_2[x_{\alpha_{q-1}}^{c_{q-1}},x_{\beta_{q-1}}^{d_{q-1}}]).
$$
We note that if $i=2$ and $\alpha_{q-1}=\beta_{q-1}$ then
\begin{equation}
\label{w_i[] short}
\mathbf w_i[x_{\alpha_{q-1}}^{c_{q-1}},x_{\beta_{q-1}}^{d_{q-1}}]=\biggl(\prod_{\ell=1}^{2k} x_{\alpha_{q-1}+c_{q-1}}^{(\ell)}\biggr)\cdot\biggl(\prod_{j=c_{q-1}+1}^{d_{q-1}-1} x_{\alpha_{q-1}}^{(j)}\biggl(\prod_{\ell=1}^{2k} x_{\alpha_{q-1}+j}^{(\ell)}\biggr)\biggr),
\end{equation}
while if $i=2$ and $\alpha_{q-1}<\beta_{q-1}$ then
\begin{equation}
\label{w_i[] long}
\begin{aligned}
&\mathbf w_i[x_{\alpha_{q-1}}^{c_{q-1}},x_{\beta_{q-1}}^{d_{q-1}}]=\biggl(\prod_{\ell=1}^{2k} x_{\alpha_{q-1}+c_{q-1}}^{(\ell)}\biggr)\cdot\biggl(\prod_{j=c_{q-1}+1}^{2k} x_{\alpha_{q-1}}^{(j)}\biggl(\prod_{\ell=1}^{2k} x_{\alpha_{q-1}+j}^{(\ell)}\biggr)\biggr)\\[-3pt]
&\cdot\biggl(\prod_{\alpha_{q-1}<\gamma<\beta_{q-1}}\biggl(\prod_{j=1}^{2k} x_{\gamma}^{(j)}\biggl(\prod_{\ell=1}^{2k} x_{\gamma+j}^{(\ell)}\biggr)\biggr)\biggr)\cdot\biggl(\prod_{j=1}^{d_{q-1}-1} x_{\beta_{q-1}}^{(j)}\biggl(\prod_{\ell=1}^{2k} x_{\beta_{q-1}+j}^{(\ell)}\biggr)\biggr).
\end{aligned}
\end{equation}
By the induction assumption, the endomorphism $\xi$ maps all letters located between the first occurrences of $x_{\alpha_{q-1}}^{(c_{q-1})}$ and $x_{\beta_{q-1}}^{(d_{q-1})}$ in $\mathbf w$ into the empty word, i.e., 
$$
\xi(\mathbf w_1[x_{\alpha_{q-1}}^{(c_{q-1})},x_{\beta_{q-1}}^{(d_{q-1})}])=\lambda.
$$
Then, taking into account the equalities~\eqref{w_i[] short} and~\eqref{w_i[] long}, we have that if $\alpha_{q-1}=\beta_{q-1}$ then the equality
\begin{equation}
\label{xi(w_i[]) short}
\xi(\mathbf w_i[x_{\alpha_{q-1}}^{c_{q-1}},x_{\beta_{q-1}}^{d_{q-1}}])=\prod_{j=c_{q-1}}^{d_{q-1}-1}\biggl(\prod_{\ell=1}^{2k} \xi(x_{\alpha_{q-1}+j}^{(\ell)})\biggr)
\end{equation}
with $i=2$ is true, while if $\alpha_{q-1}<\beta_{q-1}$ then the equality
\begin{equation}
\label{xi(w_i[]) long}
\begin{aligned}
&\xi(\mathbf w_i[x_{\alpha_{q-1}}^{c_{q-1}},x_{\beta_{q-1}}^{d_{q-1}}])=\biggl(\prod_{j=c_{q-1}}^{2k} \biggl(\prod_{\ell=1}^{2k} \xi(x_{\alpha_{q-1}+j}^{(\ell)})\biggr)\biggr)\\
&\cdot\biggl(\prod_{\alpha_{q-1}<\gamma<\beta_{q-1}}\biggl(\prod_{j=1}^{2k} \biggl(\prod_{\ell=1}^{2k} \xi(x_{\gamma+j}^{(\ell)})\biggr)\biggr)\biggr)\cdot
\biggl(\prod_{j=1}^{d_{q-1}-1} \biggl(\prod_{\ell=1}^{2k} \xi(x_{\beta_{q-1}+j}^{(\ell)})\biggr)\biggr)
\end{aligned}
\end{equation}
with $i=2$ is true. This fact and Remark~\ref{subwords of word equivalent to b_n} imply that there exist $\alpha_q,\beta_q\in \mathcal M_{2k}^q$ and $1\le c_q,d_q\le 2k$ such that the claim~\eqref{condition for alpha_q beta_q} is true, $\zeta_{\frac{2p+q-1}{2}}(x_{2p+q})=\xi(x_{\alpha_q}^{c_q})$ and $\zeta_{\frac{2p+q-1}{2}}(y_{2p+q})=\xi(x_{\beta_q}^{d_q})$. Since the first occurrence of $x_{\alpha_q}^{c_q}$ precedes the first occurrence of $x_{\beta_q}^{d_q}$ in $\mathbf a_k$, we obtain that the first occurrence of $\xi(x_{\alpha_q}^{c_q})$ precedes the first occurrence of $\xi(x_{\beta_q}^{d_q})$ in $\mathbf w$, whence $\xi(x_{\alpha_q}^{c_q})=\zeta_{\frac{2p+q-1}{2}}(x_{2p+q})=x_{2p+q}$ and $\xi(x_{\beta_q}^{d_q})=\zeta_{\frac{2p+q-1}{2}}(y_{2p+q})=y_{2p+q}$. So, we have proved the equality~\eqref{xi(x_2p+q) and xi(y_2p+q)} for all odd $q$. 

Suppose now that $q$ is even. Then $x_{2p+q}y_{2p+q}=\xi(\mathbf w_1[x_{\alpha_{q-1}}^{c_{q-1}},x_{\beta_{q-1}}^{d_{q-1}}])$. Note that if $\alpha_{q-1}=\beta_{q-1}$ then the equality~\eqref{w_i[] short} is true whenever $i=1$ and if $\alpha_{q-1}<\beta_{q-1}$ then the equality~\eqref{w_i[] long} with $i=1$ is true. By the induction assumption, the endomorphism $\xi$ maps all letters located between the second occurrences of $x_{\alpha_{q-1}}^{(c_{q-1})}$ and $x_{\beta_{q-1}}^{(d_{q-1})}$ in $\mathbf w$ into the empty word, i.e. $\xi(\mathbf w_2[x_{\alpha_{q-1}}^{(c_{q-1})},x_{\beta_{q-1}}^{(d_{q-1})}])=\lambda$. Then, taking into account the equalities~\eqref{w_i[] short} and~\eqref{w_i[] long}, we have that if $\alpha_{q-1}=\beta_{q-1}$ then the equality~\eqref{xi(w_i[]) short} with $i=1$ is true and if $\alpha_{q-1}<\beta_{q-1}$ then the equality~\eqref{xi(w_i[]) long} with $i=1$ is true. This fact and Remark~\ref{subwords of word equivalent to b_n} imply that there exist $\alpha_q,\beta_q\in \mathcal M_{2k}^q$ and $1\le c_q,d_q\le 2k$ such that the claim~\eqref{condition for alpha_q beta_q} is true, $x_{2p+q}=\xi(x_{\alpha_q}^{c_q})$ and $y_{2p+q}=\xi(x_{\beta_q}^{d_q})$.

Thus, we have shown that, for any $1\le q\le 2n-2p$, there are letters $x_{\alpha_q}^{c_q}$ and $x_{\beta_q}^{d_q}$ such that $\alpha_q,\beta_q\in \mathcal M_{2k}^q$ and the claims~\eqref{xi(x_2p+q) and xi(y_2p+q)} and~\eqref{condition for alpha_q beta_q} are true. In particular, $x_{2n}=\xi(x_{\alpha_{2n-2p}}^{c_{2n-2p}})$ and $y_{2n}=\xi(x_{\beta_{2n-2p}}^{d_{2n-2p}})$. It follows that $\xi(\mathbf w_2[x_{\alpha_{2n-2p}}^{c_{2n-2p}},x_{\beta_{2n-2p}}^{d_{2n-2p}}])=t_{2n}$. If $k>n$ then $\con(\mathbf w_2[x_{\alpha_{2n-2p}}^{c_{2n-2p}},x_{\beta_{2n-2p}}^{d_{2n-2p}}])\subseteq\mul(\mathbf a_k)$. This contradicts the fact that $t_{2n}\in\simple(\mathbf w)$.
\end{proof}

\section{Proof of the main result}
\label{proof of theorem}

(i) We are going to verify that the lattice $L(\mathbf N\vee\overleftarrow{\mathbf M})$  ''modulo'' the interval $[\mathbf M\vee\overleftarrow{\mathbf M},\mathbf N\vee\overleftarrow{\mathbf M}]$ has the form shown in Fig.~\ref{L(N vee dual M)}. In view of~\cite[Proposition~5.2]{Gusev-Vernikov-18} and~\cite[Fig.~1]{Jackson-05}, the lattices $L(\mathbf N)$ and $L(\mathbf M\vee\overleftarrow{\mathbf M})$ have the form shown in Fig.~\ref{L(N vee dual M)}. Let $\mathbf V$ be a proper subvariety of the variety $\mathbf N\vee\overleftarrow{\mathbf M}$ which is not contained in $\mathbf N$ and $\mathbf M\vee\overleftarrow{\mathbf M}$. We need to check that $\mathbf V$ belongs to the interval $[\mathbf M\vee\overleftarrow{\mathbf M},\mathbf N\vee\overleftarrow{\mathbf M}]$.

A variety of monoids is said to be \emph{completely regular} if it consists of \emph{completely regular monoids}~(i.e., unions of groups). If the variety $\mathbf V$ is completely regular then  it is a variety of \emph{bands} (i.e. idempotent monoids) because it satisfies the identity
\begin{equation}
\label{xx=xxx}
x^2\approx x^3.
\end{equation} 
Evidently, every variety of bands which satisfies the identity
\begin{equation}
\label{xxy=yxx}
x^2y\approx yx^2,
\end{equation} 
 is commutative. Therefore, $\mathbf V$ is one of the varieties $\mathbf T$ or $\mathbf{SL}$, a contradiction. So, $\mathbf V$ is non-completely regular. Suppose that $\mathbf D_2 \nsubseteq \mathbf V$. Then it follows from~\cite[Lemma 2.15]{Gusev-Vernikov-18} that $\mathbf V$ satisfies the identity
\begin{equation}
\label{xyx=x^qyx^r}
xyx\approx x^qyx^r
\end{equation}
where either $q>1$ or $r>1$. If $q>1$ then $\mathbf V$ satisfies the identities
$$
xyx\stackrel{\eqref{xyx=x^qyx^r}}\approx x^qyx^r\stackrel{\eqref{xx=xxx}}\approx x^2yx^r\stackrel{\eqref{xxy=yxx}}\approx yx^{2+r}\stackrel{\eqref{xx=xxx}}\approx yx^2\stackrel{\eqref{xxy=yxx}}\approx x^2y,
$$
whence $\mathbf V\subseteq \mathbf D_1$, a contradiction. If $r>1$ then the identities
$$
xyx\stackrel{\eqref{xyx=x^qyx^r}}\approx x^qyx^r\stackrel{\eqref{xx=xxx}}\approx x^qyx^2\stackrel{\eqref{xxy=yxx}}\approx x^{2+q}y\stackrel{\eqref{xx=xxx}}\approx x^2y\stackrel{\eqref{xxy=yxx}}\approx yx^2
$$
hold in $\mathbf V$. We obtain a contradiction again. Thus, $\mathbf D_2 \subseteq \mathbf V$. 

If $\mathbf M\nsubseteq\mathbf V$ then it follows from~\cite[Lemma~4.9(i)]{Gusev-Vernikov-18} that the variety $\mathbf V$ satisfies the identity $\sigma_1$. Therefore, $\mathbf V\subseteq \overleftarrow{\mathbf N}$. The lattice $L(\overleftarrow{\mathbf N})$ is isomorphic to the lattice $L(\mathbf N)$. Besides that, $\mathbf V\ne \overleftarrow{\mathbf N}$. This implies the wrong inclusion $\mathbf V\subseteq \overleftarrow{\mathbf M}\subset \mathbf M\vee\overleftarrow{\mathbf M}$. By symmetry, if $\overleftarrow{\mathbf M}\nsubseteq\mathbf V$ then $\mathbf V\subseteq\mathbf N$. This inclusion is also impossible. Thus, we have shown that $\mathbf M\vee\overleftarrow{\mathbf M}\subseteq \mathbf V$. Therefore, $\mathbf V\in[\mathbf M\vee\overleftarrow{\mathbf M},\mathbf N\vee\overleftarrow{\mathbf M}]$.

(ii) We denote an identity basis of the variety $\mathbf N\vee\overleftarrow{\mathbf M}$ by $\Sigma$. Let $\mathrm K$ be a subset of $\mathbb N$. Put $\Sigma_{\mathrm K}=\{\mathbf a_n\approx \mathbf a_n'\mid n \in \mathrm K\}$. We are going to verify that different subsets of the form $\Sigma_{\mathrm K}$ define different subvarieties within the variety $\mathbf N\vee\overleftarrow{\mathbf M}$. Arguing by contradiction, suppose that there are $n$ and $\mathrm K\subseteq \mathbb N$ such that $n\notin \mathrm K$ and the identity $\mathbf a_n\approx \mathbf a_n'$ follows from the identity system $\Sigma\cup \Sigma_{\mathrm K}$. Then there exists a sequence of words $\mathbf a_n= \mathbf w_0,\mathbf w_1,\dots,\mathbf w_m= \mathbf a_n'$ such that, for any $i\in\{0,1,\dots,m-1\}$ there exist words $\mathbf u_i,\mathbf v_i\in F^1$, an endomorphism $\xi_i$ of $F^1$ and an identity $\mathbf p_i\approx \mathbf q_i\in \Sigma\cup \Sigma_{\mathrm K}$ such that $\mathbf w_i= \mathbf u_i\xi_i(\mathbf p_i)\mathbf v_i$ and $\mathbf w_{i+1}= \mathbf u_i\xi_i(\mathbf q_i)\mathbf v_i$. We can assume without loss of generality that $\mathbf w_i\ne\mathbf w_{i+1}$ for all $i\in\{0,1,\dots,m-1\}$. The words $\mathbf a_n$ and $\mathbf a_n'$ are $1$-equivalent by Remark~\ref{rm s-decompositions of a_n,a_n'}. But the $1$-block $xy$ of $\mathbf a_n$ does not coincide with the corresponding $1$-block $yx$ of $\mathbf a_n'$. Then Lemma~\ref{word problem N} implies that the variety $\mathbf N$ violates the identity $\mathbf a_n\approx \mathbf a_n'$, whence there is a number $r\in\{0,1,\dots,m-1\}$ such that $\mathbf p_r\approx \mathbf q_r$ equals one of the identities $\mathbf a_k\approx \mathbf a_k'$ or $\mathbf a_k'\approx \mathbf a_k$ for some $k\ne n$. Let $r$ be the least number with such a property. Then the identity $\mathbf a_n\approx \mathbf w_r$ holds in the variety $\mathbf N\vee\overleftarrow{\mathbf M}$. In view of Remark~\ref{rm s-decompositions of a_n,a_n'}, the full decomposition of the word $\mathbf a_n$ has the form~\eqref{s-decompositions of a_n,a_n'} with $\chi(xy)=xy$. Then Lemma~\ref{word problem N} and the dual to Lemma~\ref{word problem M} imply that the word $\mathbf w_r$ coincides with~\eqref{word equivalent to a_n} where the equality~\eqref{equalities for a_n} is true for all odd $1\le\ell\le 2n-1$ and for all $\alpha\in \mathcal M_{2n}^\ell$.

We note that $\xi_r(x)\ne\lambda$ and $\xi_r(y)\ne\lambda$ because the identity $\mathbf w_r\approx \mathbf w_{r+1}$ is non-trivial. This implies that the length of the word $\xi_r(xy)$ is more than $2$. In view of Remark~\ref{subwords of word equivalent to a_n}, the words $\xi_r(x)$ and $\xi_r(y)$ are letters. Since the first occurrences of the letters $\xi_r(x)$ and $\xi_r(y)$ occur in the words $\mathbf w_r$ and $\mathbf w_{r+1}$ in the opposite order, Lemma~\ref{word problem M} implies that some subblock of some $1$-block of the word $\mathbf w_r$ contains the subword $\xi_r(xy)$ whenever $\mathbf p_r=\mathbf a_k$, and the subword $\xi_r(yx)$ whenever $\mathbf p_r=\mathbf a_k'$. Recall that the full decomposition of the word $\mathbf w_r$ has the form~\eqref{word equivalent to a_n}. Therefore, if the identity $\mathbf p_r\approx \mathbf q_r$ equals $\mathbf a_k'\approx \mathbf a_k$ then either $\xi_r(yx)=xy$ or $\xi_r(yx)=x_\gamma^{(p)}x_\gamma^{(p+1)}$ for some $1\le p< 2n$, $1\le h<n$ and $\gamma\in \mathcal M_{2n}^{2h}$. Then either $\xi_r(y)=x$ and $\xi_r(x)=y$ or $\xi_r(y)=x_\gamma^{(p)}$ and $\xi_r(x)=x_\gamma^{(p+1)}$. Since the second occurrence of $x$ precedes the second occurrence of $y$ in $\mathbf a_k'$, we have that the second occurrence of $\xi(x)$ precedes the second occurrence of $\xi(y)$ in $\mathbf w_r$. But this is impossible, because the second occurrence of $y$ is preceded by the second occurrence of $x$ in $\mathbf w_r$, while the second occurrence of $x_\gamma^{(p+1)}$ is preceded by the second occurrence of $x_\gamma^{(p)}$ in $\mathbf w_r$. So, the identity $\mathbf p_r\approx \mathbf q_r$ cannot coincide with the identity $\mathbf a_k'\approx \mathbf a_k$ and, therefore, $\mathbf p_r\approx \mathbf q_r$ equals $\mathbf a_k\approx \mathbf a_k'$.

Suppose that $k<n$. Since the full decomposition of the word $\mathbf w_r$ has the form~\eqref{word equivalent to a_n}, either $\xi_r(xy)=xy$ or $\xi_r(xy)=x_\gamma^{(p)}x_\gamma^{(p+1)}$ for some $1\le p< 2n$, $1\le h<n$ and $\gamma\in \mathcal M_{2n}^{2h}$. If $\xi(xy)=xy$ then
$$
\prod_{i=1}^{2k} \xi_r(x_1^{(i)})=\prod_{i=1}^{2n} \zeta_1(x_1^{(i)}).
$$
Since $k<n$, there is $1\le i\le 2k$ such that the length of the word $\xi(x_1^{(i)})$ is more than $1$. This contradicts the claim~\eqref{equalities for a_n} and Remark~\ref{subwords of word equivalent to a_n}. If $\xi_r(xy)=x_\gamma^{(p)}x_\gamma^{(p+1)}$ for some $1\le p< 2n$, $1\le h<n$ and $\gamma\in \mathcal M_{2n}^{2h}$ then
$$
\prod_{i=1}^{2k} \xi_r(x_1^{(i)})=\prod_{i=1}^{2n} \zeta_{\gamma+p}(x_{\gamma+p}^{(i)}).
$$
Taking into account the claim~\eqref{equalities for a_n}, we get a contradiction with Remark~\ref{subwords of word equivalent to a_n}.

Suppose now that $k>n$. We need some more notation. The smallest element of the set $\mathcal M_s^t$ is denoted by $\gamma_s^t$. Further, we denote by $\eta$ the endomorphism which is defined by the following equalities:
\begin{align*}
\eta(x)=x_0,\ \eta(y)=y_0,&\ \eta(s)=s_0,\ \eta(s_{q'})=s_{q'},\\
\eta(x_{\gamma_{2n}^q}^{(1)})=x_q,\ \eta(x_{\gamma_{2n}^q}^{(2)})=y_q,&
\ \eta(s_{\gamma_{2n}^{2n}}^{(1)})=s_{2n},\ \eta(s_{\gamma_{2n}^{2n}}^{(2)})=t_{2n},\\
\eta(z)&=\lambda,
\end{align*}
where $1\le q\le 2n$, $1\le q'\le 2n-1$ and $z$ is an arbitrary letter which differs from $x$, $y$, $s$, $s_{q'}$, $x_{\gamma_{2n}^q}^{(1)}$, $x_{\gamma_{2n}^q}^{(2)}$, $s_{\gamma_{2n}^{2n}}^{(1)}$ and $s_{\gamma_{2n}^{2n}}^{(2)}$.

We note that the word $\eta(\mathbf w_r)$ equals the word~\eqref{word equivalent to b_n} for some endomorphisms $\zeta_1,\zeta_2,\dots, \zeta_{n-1}$ of $F^1$ such that the equality~\eqref{equalities for b_n} is true for all $0\le\ell\le n-1$. Obviously, $\eta(\mathbf w_r)=\eta(\mathbf u_r)\eta(\xi_r(\mathbf p_r))\eta(\mathbf v_r)$ and $\eta(\mathbf w_{r+1})=\eta(\mathbf u_r)\eta(\xi_r(\mathbf q_r))\eta(\mathbf v_r)$. Note that the identity $\eta(\mathbf w_r)\approx \eta(\mathbf w_{r+1})$ is non-trivial because the first occurrences of the letters $\eta(\xi_r(x))$ and $\eta(\xi_r(y))$ occur in the words $\eta(\xi_r(\mathbf p_r))$ and $\eta(\xi_r(\mathbf q_r))$ in the opposite order. Then we have a contradiction with Lemma~\ref{substitution a_k in b_n} and inequality $n<k$. So, we have proved that different subsets of the form $\Sigma_{\mathrm K}$ define different subvarieties within the variety $\mathbf N\vee\overleftarrow{\mathbf M}$. All these subvarieties belongs to the interval $[\mathbf M\vee\overleftarrow{\mathbf M},\mathbf N\vee\overleftarrow{\mathbf M}]$ by Lemma~\ref{word problem M}, the dual to Lemma~\ref{word problem M} and Remark~\ref{rm s-decompositions of a_n,a_n'}. This implies that the lattice of all subsets of $\mathbb N$ order-embeds into the interval $[\mathbf M\vee\overleftarrow{\mathbf M},\mathbf N\vee\overleftarrow{\mathbf M}]$. It is well known that the lattice of all subsets of $\mathbb N$ is uncountable and violates the ascending chain condition and the descending chain condition. Theorem~\ref{main result} is proved.\qed

\smallskip

The following statement provide some more new examples of finitely generated varieties of monoids with continuum many subvarieties.

\begin{corollary}
\label{D_2 join G}
Let $G$ be a finite group that does not satisfy the identities $xytxy\approx yxtyx$ and $xytyx\approx yxtxy$. Then the monoid $S(xtx)\times G$ generates a variety with continuum many subvarieties.
\end{corollary}

\begin{proof}
Let $\mathbf G$ denote the monoid variety generated by $G$. We note that the monoid $S(xtx)$ generates the variety $\mathbf D_2$~\cite{Jackson-Sapir-00}. It follows that $\mathbf D_2\vee \mathbf G$ is generated by $S(xtx)\times G$. The word $xtx$ is an isoterm for $\mathbf D_2\vee \mathbf G$ by Lemma~\ref{S(W) in V}. Then~\cite[Fact~3.1(ii)]{Sapir-15} implies that if the variety $\mathbf D_2\vee \mathbf G$ satisfies a non-trivial identity $xysxty\approx \mathbf w$ then $\mathbf w=yxsxty$. But every group that satisfies the identity $\sigma_1$ is Abelian one, while it is evident that the group $G$ is non-Abelian. Hence $xysxty$ is an isoterm for $\mathbf D_2\vee \mathbf G$. According to Lemma~\ref{S(W) in V}, $S(xysxty)\in \mathbf D_2\vee \mathbf G$. Analogously, $S(xsytxy)\in \mathbf D_2\vee \mathbf G$. Further, since word $xtx$ is an isoterm for the variety $\mathbf D_2\vee \mathbf G$, if this variety satisfies an identity $xytxy\approx \mathbf v$ then $\mathbf v\in\{xytxy,xytyx,yxtxy,yxtyx\}$. The word $\mathbf v$ cannot coincide with $yxtyx$ because $G$ violates the identity $xytxy\approx yxtyx$. Let $m$ denote the exponent of $G$. If $\mathbf v=xytyx$, then $G$ satisfies the identities
\[
xy\approx (xy)^{m+1}\approx (xy)^m(yx)\approx yx
\]
contradicting the fact that $G$ is a non-Abelian group.
If $\mathbf v=yxtxy$, then $G$ satisfies the identities
\[
xy\approx (xy)^{m+1}\approx (yx)(xy)^m\approx yx
\]
contradicting the fact that the group $G$ is a non-Abelian again. Therefore, $\mathbf v=xytxy$. We see that $xytxy$ is an isoterm for $\mathbf D_2 \vee \mathbf G$. By similar arguments we can show that $xytyx$ is an isoterm for $\mathbf D_2\vee \mathbf G$ as well. Now Lemma~\ref{S(W) in V} applies, yielding that $S(xytxy,xytyx)\in\mathbf D_2\vee \mathbf G$.

In view of the dual to Lemma~\ref{generator of M}, $\overleftarrow{\mathbf M}$ is generated by $S(xsytxy)$, whence $\overleftarrow{\mathbf M}\subseteq \mathbf D_2\vee \mathbf G$. It follows from the dual to Example~1 in Erratum to~\cite{Jackson-05} that $\mathbf N$ is generated by the monoid $S(xysxty)$ together with some quotient monoid of $S(xytxy,xytyx)$. It follows that $\mathbf N\subseteq \mathbf D_2\vee \mathbf G$. Therefore, $\mathbf N\vee\overleftarrow{\mathbf M}\subseteq \mathbf D_2\vee \mathbf G$. Theorem~\ref{main result} implies that $\mathbf D_2\vee \mathbf G$ contains continuum many subvarieties.
\end{proof}

In conclusion, we note that Fig.~\ref{L(N vee dual M)} and Theorem~\ref{main result}(ii) imply that the lattice $L(\mathbf N\vee\overleftarrow{\mathbf M})$ is non-modular. The following question seems to be interesting.
\begin{question}
\label{question non-trivial}
Does the lattice $L(\mathbf N\vee\overleftarrow{\mathbf M})$ satisfy any non-trivial identity?
\end{question}

{\color{red}The previous version of this paper has been published in Siberian Electronic Math. Reports \textbf{16} (2019) 983--997. In this version, there is an error in the proof of Corollary~\ref{D_2 join G}. Namely, the following is claimed there: $\mathbf N\vee\overleftarrow{\mathbf M}$ is a subvariety of the variety generated by $S(xtx)\times G$ for any finite non-Abelian group $G$.
In fact, this result is wrong in general. 
For example, it is easy to see that the quaternion group
\[
Q_8=\langle i, j,k \mid i^2=j^2=k^2 = ijk\rangle
\]
satisfies the identity $xytxy\approx yxtyx$, whence $\mathbf N$ is not contained in the variety generated by $S(xtx)\times Q_8$.
As we have seen above, the discussed result is true whenever $G$ is a finite group  violated the identities $xytxy\approx yxtyx$ and $xytyx\approx yxtxy$.
}

\subsection*{Acknowledgments.} The author is sincerely grateful to Professor Vernikov for his attention and assistance in the writing of the article and to the anonymous referee for several useful remarks.


\begin{thebibliography}{99}
\label{bibl}
\bibitem{Evans-71}
T. Evans, {\it The lattice of semigroup varieties}, Semigroup Forum, \textbf2 (1971), 1--43.
\bibitem{Gusev-18-IzVUZ}
S.V. Gusev, {\it On the lattice of overcommutative varieties of monoids}, Izv. Vyssh. Uchebn. Zaved. Matem., \textbf5 (2018), 28--32 [Russian; Engl. translation: Russian Math. Iz. VUZ, \textbf{62}:5 (2018), 23--26].
\bibitem{Gusev-18-AU}
S.V. Gusev, {\it Special elements of the lattice of monoid varieties}, Algebra Univ., \textbf{97}:2, Article 29 (2018), 1--12.
\bibitem{Gusev-Vernikov-18}
S.V. Gusev, B.M. Vernikov, {\it Chain varieties of monoids}, Dissert. Math., \textbf{534} (2018), 1--73.
\bibitem{Head-68}
T.J. Head, {\it The varieties of commutative monoids}, Nieuw Arch. Wiskunde. III~Ser., \textbf{16} (1968), 203--206.
\bibitem{Jackson-05}
M. Jackson, {\it Finiteness properties of varieties and the restriction to finite algebras}, Semigroup Forum, \textbf{70}:2 (2005), 154--187; {\it Erratum to: Finiteness properties of varieties and the restriction to finite algebras}, Semigroup Forum, \textbf{96}:1 (2018), 197--198.
\bibitem{Jackson-Lee-18}
M. Jackson, E.W.H. Lee, {\it Monoid varieties with extreme properties}, Trans. Amer. Math. Soc., \textbf{370} (2018), 4785--4812.
\bibitem{Jackson-Sapir-00}
M. Jackson, O. Sapir, {\it Finitely based, finite sets of words}, Int. J. Algebra and Comput., \textbf{10} (2000), 683--708.
\bibitem{Kozhevnikov-12}
P.A. Kozhevnikov, {\it On nonfinitely based varieties of groups of large prime exponent}, Commun. in Algebra, \textbf{40} (2012), 2628--2644.
\bibitem{Lee-08}
E.W.H. Lee, {\it On the variety generated by some monoid of order five}, Acta Scientiarum Mathematicarum, \textbf{74}:3--4 (2008), 509--537.
\bibitem{Lee-12}
E.W.H. Lee, {\it Varieties generated by $2$-testable monoids}, Studia Sci. Math. Hungar., \textbf{49} (2012), 366--389.
\bibitem{Lee-14}
E.W.H. Lee, {\it Inherently non-finitely generated varieties of aperiodic monoids with central idempotents}, Zapiski Nauchnykh Seminarov POMI (Notes of Scientific Seminars of the St.~Petersburg Branch of the Math. Institute of the Russ. Acad. of Sci.), \textbf{423} (2014), 166--182.
\bibitem{Lvov-73}
I.V. L'vov, {\it Varieties of associative rings}.~I, Algebra i Logika, \textbf{12} (1973), 269--297 [Russian; Engl. translation: Algebra and Logic, \textbf{12} (1973), 150--167].
\bibitem{Perkins-69}
P. Perkins, {\it Bases for equational theories of semigroups}, J. Algebra, \textbf{11} (1969), 298--314.
\bibitem{Pollak-81}
Gy. Poll\'ak, {\it Some lattices of varieties containing elements without cover}, Quad. Ric. Sci., \textbf{109} (1981), 91--96.
\bibitem{Sapir-91}
M. Sapir, {\it On cross semigroup varieties and related questions}, Semigroup Forum, \textbf{42} (1991), 345--364.
\bibitem{Sapir-15}
O. Sapir, {\it Non-finitely based monoids}, Semigroup Forum, \textbf{90}:3 (2015), 557--586.
\bibitem{Shevrin-Vernikov-Volkov-09}
L.N. Shevrin, B.M. Vernikov, M.V. Volkov, {\it Lattices of semigroup varieties}, Izv. Vyssh. Uchebn. Zaved. Matem., \textbf3 (2009), 3--36 [Russian; Engl. translation: Russian Math. Iz. VUZ, \textbf{53}:3 (2009), 1--28].
\bibitem{Wismath-86}
S.L. Wismath, {\it The lattice of varieties and pseudovarieties of band monoids}, Semigroup Forum, \textbf{33} (1986), 187--198.
\end{thebibliography}
\end{document}